\documentclass[11pt]{amsart}
\usepackage{latexsym,amsmath,amssymb,epsfig,amsthm}
\usepackage{graphicx}
\usepackage{tikz}
\usetikzlibrary{calc}
\usepackage[enableskew,vcentermath]{youngtab}
\usepackage{hyperref}
\hypersetup{colorlinks=false}
\input epsf
\newtheorem{theorem}{Theorem}[section]
\newtheorem{proposition}[theorem]{Proposition}
\newtheorem{corollary}[theorem]{Corollary}

\newtheorem{example}[theorem]{Example}
\newtheorem{lemma}[theorem]{Lemma}

\newtheorem{question}[theorem]{Question}
\newcommand{\asc}{{\rm asc}}

\newcommand{\des}{{\rm des}}

\newcommand{\esd}{{\rm esd}}

\newcommand{\exc}{{\rm exc}}

\newcommand{\fix}{{\rm fix}}
\newcommand{\Fix}{{\rm Fix}}

\newcommand{\sd}{{\rm sd}}

\newcommand{\supp}{{\rm supp}}

\newcommand{\bB}{{\mathcal B}}
\newcommand{\dD}{{\mathcal D}}

\newcommand{\fF}{{\mathcal F}}

\newcommand{\iI}{{\mathcal I}}

\newcommand{\sS}{{\mathcal S}}
\newcommand{\wW}{{\mathcal W}}

\newcommand{\RR}{{\mathbb R}}
\newcommand{\fS}{{\mathfrak S}}
\newcommand{\NN}{{\mathbb N}}

\renewcommand{\to}{\rightarrow}

\begin{document}
\title[Local $h$-polynomials and real-rootedness]
{Local $h$-polynomials, uniform triangulations
and real-rootedness}

\author{Christos~A.~Athanasiadis}

\address{Department of Mathematics\\
National and Kapodistrian University of Athens\\
Panepistimioupolis\\ 15784 Athens, Greece}
\email{caath@math.uoa.gr}

%\date{\today}
\thanks{Research supported by the Hellenic 
Foundation for Research and Innovation (H.F.R.I.) 
under the `2nd Call for H.F.R.I. Research Projects 
to support Faculty Members \& Researchers' 
(Project Number: HFRI-FM20-04537).}
\thanks{ \textit{Mathematics Subject 
Classifications}: Primary: 05E45; 
                  Secondary: 05A15, 26C10}
\thanks{\textit{Key words and phrases}. 
Triangulation, barycentric subdivision,  
edgewise subdivision, local $h$-polynomial, 
real-rooted polynomial.}

\begin{abstract}
The local $h$-polynomial was introduced by Stanley 
as a fundamental enumerative invariant of a 
triangulation $\Delta$ of a simplex. This 
polynomial is known to have nonnegative and 
symmetric coefficients and is conjectured to be 
$\gamma$-positive when $\Delta$ is flag. This 
paper shows that the local $h$-polynomial has the 
stronger property of being real-rooted when 
$\Delta$ is the barycentric subdivision of an 
arbitrary geometric triangulation $\Gamma$ of the 
simplex. An analogous result for edgewise 
subdivisions is proven. The proofs are based on a 
new combinatorial 
formula for the local $h$-polynomial of $\Delta$, 
which is valid when $\Delta$ is any uniform 
triangulation of $\Gamma$. A combinatorial 
interpretation of the local $h$-polynomial of the 
second barycentric subdivision of the simplex is
deduced.
\end{abstract}

\maketitle

\section{Introduction}
\label{sec:intro}
 
The local $h$-polynomial of a triangulation 
$\Gamma$ of a simplex was introduced by Stanley as 
a powerful tool in his enumerative theory of 
subdivisions of simplicial complexes \cite{Sta92}. 
It is defined by the formula  
\begin{equation} \label{eq:localh-def}
  \ell_V (\Gamma, x) \, = \sum_{F \subseteq V} 
  (-1)^{n - |F|} \, h(\Gamma_F, x),
\end{equation}
where $V$ is the set of vertices of the simplex (the
abstract simplex $2^V$) triangulated by 
$\Gamma$ and $h(\Gamma_F, x)$ is the $h$-polynomial
of the restriction $\Gamma_F$ of $\Gamma$ to the face
$F$ of $2^V$ (see Section~\ref{subsec:complexes} for 
any undefined terminology and for more information 
about these concepts). The sequence of coefficients 
of $\ell_V (\Gamma, x)$ is the local $h$-vector of 
$\Gamma$. The significance of local $h$-polynomials 
stems from their appearance in Stanley's locality 
formula \cite[Theorem~3.2]{Sta92}, which expresses the 
$h$-polynomial of a triangulation of a pure simplicial 
complex as a sum of local contributions, one for every 
face of that complex. Local $h$-polynomials appear 
prominently in algebraic and geometric combinatorics 
\cite{Ath12, Ath14, Ath22, Ath23+, GS20, KS16, Sta92} 
and other areas of mathematics, such as algebraic and 
arithmetic geometry~\cite{CMM18, KS16, LPS22+, Sta17}. 
They have also been studied extensively because of 
their own independent interest \cite{Ath16b, AS12, 
Cha94, MGPSS20, KMS19, LPS23, Zh19}.

The remarkable combinatorial properties of local 
$h$-polynomials (see~\cite{Ath16a} for a short 
survey) are reminiscent 
of those of $h$-polynomials of triangulations of 
spheres. Most importantly, $\ell_V(\Gamma, x)$ has 
nonnegative and symmetric coefficients for every 
triangulation $\Gamma$ of the simplex $2^V$ 
\cite{Sta92}. Moreover, it is unimodal for every 
regular triangulation $\Gamma$ of $2^V$ \cite{Sta92} 
and is conjectured to be $\gamma$-positive whenever
$\Gamma$ is flag \cite{Ath12}. 

Gamma-positivity is a property of polynomials with 
nonnegative and symmetric coefficients which is 
implied by real-rootedness; see 
\cite[Section~7.3]{Bra15}. As is the case with 
flag triangulations of spheres and their  
$h$-polynomials~\cite{Ga05}, not all flag 
triangulations of simplices have a real-rooted local 
$h$-polynomial~\cite{Ath16b}. However, one expects
$\ell_V (\Gamma, x)$ to be real-rooted for 
sufficiently nice flag triangulations $\Gamma$, so
it makes sense to ask the following general 
question.

\begin{question} \label{que:main} 
Which flag triangulations of the simplex have a
real-rooted local $h$-polynomial?
\end{question}

Flag triangulations of the simplex which are known 
to have real-rooted local $h$-polynomials include 
the barycentric subdivision and its $r$-colored
generalization \cite[Section~3.3]{BS21} 
\cite[Section~5.1]{GS20}, the cluster 
subdivision~\cite{Zh18} and the $r$-fold edgewise 
subdivision~\cite{Zh19} of the simplex. The first 
main result of this paper exhibits a large class 
of flag triangulations of the simplex which have 
this property; it can be viewed as a local 
$h$-polynomial analogue of a result of Brenti and 
Welker~\cite{BW08}, which implies that barycentric 
subdivisions of triangulations of spheres have 
real-rooted $h$-polynomials, strengthens a special
case of a result of Juhnke-Kubitzke, Murai and Sieg
\cite{KMS19}, stating that local $h$-polynomials 
of barycentric subdivisions of CW-regular subdivisions 
of the simplex are $\gamma$-positive, and answers the 
first part of \cite[Question~6.3]{Ath22+} in the 
affirmative. We denote by $A_n(x)$ the $n$th
Eulerian polynomial \cite[Section~I.4]{StaEC1},
meaning the descent enumerator for permutations 
in the symmetric group $\fS_n$, defined so that 
$A_n(0) = 1$.

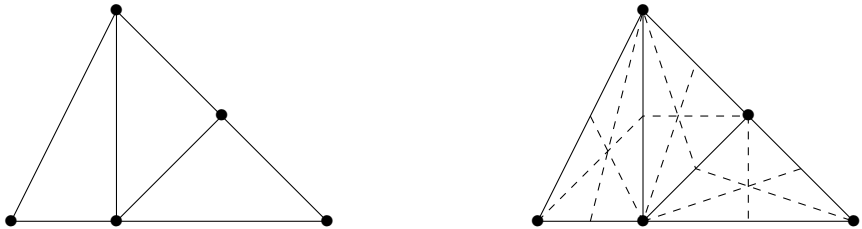
\begin{figure}
\begin{center}
\begin{tikzpicture}[scale=0.7]
\label{fg:sd}

   \draw(-4,0) node(1){$\bullet$};
   \draw(-2,0) node(2){$\bullet$};
   \draw(-2,4) node(5){$\bullet$};
   \draw(0,2) node(7){$\bullet$};
   \draw(2,0) node(8){$\bullet$};

   \draw(-4,0) -- (2,0) -- (-2,4) -- (-4,0);
   \draw(-2,4) -- (-2,0) -- (0,2);

   \draw(6,0) node(11){$\bullet$};
   \draw(8,0) node(12){$\bullet$};
   \draw(8,4) node(15){$\bullet$};
   \draw(10,2) node(17){$\bullet$};
   \draw(12,0) node(18){$\bullet$};

   \draw(6,0) -- (12,0) -- (8,4) -- (6,0);
   \draw(8,4) -- (8,0) -- (10,2);
	
	 \draw [dashed] (6,0) -- (8,2) -- (10,2) -- (10,0);
   \draw [dashed] (7,2) -- (8,0) -- (9,3);
	 \draw [dashed] (7,0) -- (8,4) -- (9,1) -- (12,0);
	 \draw [dashed] (8,0) -- (11,1);

\end{tikzpicture}
\caption{A triangulation of the 2-simplex and 
its barycentric subdivision}
\end{center}
\end{figure}

\begin{theorem} \label{thm:sd} 
Let $\Gamma$ be a triangulation of the 
$(n-1)$-dimensional simplex $2^V$ and let 
$\sd(\Gamma)$ denote its barycentric subdivision. 
Then, the local $h$-polynomial $\ell_V(\sd(\Gamma),x)$ 
is real-rooted and is interlaced by the Eulerian 
polynomial $A_n(x)$. 
\end{theorem}

Figure~\ref{fg:sd} illustrates Theorem~\ref{thm:sd}. 
We demonstrate the applicability of this theorem with 
the following nontrivial corollary. A combinatorial 
interpretation of the local $h$-polynomial of the 
second barycentric subdivision of the simplex is 
given in Section~\ref{sec:sd} (see 
Corollary~\ref{cor:2sd}), thus solving part of 
\cite[Problem~4.12]{Ath16a}.
\begin{corollary} \label{cor:sd} 
The local $h$-polynomial of the $k$th iterated 
barycentric subdivision of the $(n-1)$-dimensional 
simplex is real-rooted and is interlaced by 
$A_n(x)$ for all positive integers $n, k$. 
\end{corollary}

The second main result of this paper is an analogue
of Theorem~\ref{thm:sd} for edgewise subdivisions; 
see \cite{Ath16b} \cite[Section~3.2]{MW17}, the 
references given there and 
Section~\ref{subsec:complexes} for information 
about this important class of triangulations. The 
polynomial $E_{n,r}(x)$ which appears in the 
statement is the descent enumerator for words $w 
\in \{0, 1,\dots,r-1\}^{n-1}$ (see 
Section~\ref{sec:esd} for the precise definition).
\begin{theorem} \label{thm:esd} 
Let $\Gamma$ be a triangulation of the 
$(n-1)$-dimensional simplex $2^V$ and let 
$\esd_r(\Gamma)$ denote its $r$-fold edgewise 
subdivision. Then, the local $h$-polynomial 
$\ell_V(\esd_r(\Gamma),x)$ is real-rooted and is 
interlaced by $E_{n,r}(x)$ for every $r \ge n$. 
\end{theorem}

The assumption that $r \ge n$ cannot be dropped 
from Theorem~\ref{thm:esd} (see 
Example~\ref{ex:esdr}).

The barycentric subdivision $\sd(\Gamma)$ can be 
defined for any regular CW-complex $\Gamma$. 
Theorem~\ref{thm:sd} naturally raises the following
question. 
\begin{question} \label{que:CW} 
Is $\ell_V(\sd(\Gamma),x)$ real-rooted for every 
CW-regular subdivision $\Gamma$ of $2^V$?
\end{question}

Our strategy to prove Theorem~\ref{thm:sd} is to 
express $\ell_V(\sd(\Gamma),x)$ as a nonnegative linear 
combination of real-rooted polynomials with nonnegative
coefficients which have $A_n(x)$ as a common interleaver.
A prototypical application of this method (see 
also~\cite{CS07}) yields a proof of the aforementioned 
result of Brenti and Welker~\cite{BW08}, namely that 
barycentric subdivisions of Cohen--Macaulay simplicial 
complexes have real-rooted $h$-polynomials, as follows.
One can express (see \cite[Theorem~1]{BW08}) the 
$h$-polynomial of the barycentric subdivision of any 
$(n-1)$-dimensional simplicial complex $\Gamma$ as 
\begin{equation} \label{eq:BW08}
h(\sd(\Gamma), x) = \sum_{k=0}^n h_k(\Gamma) 
                     p_{n,k}(x), 
\end{equation}
where $(h_k(\Gamma))_{0 \le k \le n}$ is the 
$h$-vector of $\Gamma$ and $p_{n,k}(x)$ is the 
descent enumerator of permutations $w \in \fS_{n+1}$ 
such that $w(1) = k+1$ (see Section~\ref{sec:sd}). 
The polynomials $p_{n,k}(x)$ satisfy a standard 
recurrence, from which one may deduce that 
$(p_{n,k}(x))_{0 \le k \le n}$ is an interlacing 
sequence of real-rooted polynomials for every $n \in 
\NN$; see \cite[Example~7.8.8]{Bra15}. As a 
consequence, $p_{n,k}(x)$ is interlaced by $p_{n,0}
(x) = A_n(x)$ for every $k$ and hence so is
$h(\sd(\Gamma), x)$, provided that $\Gamma$ has a 
nonnegative $h$-vector. 

The analogues of Equation~(\ref{eq:BW08}) which 
are needed to prove Theorems~\ref{thm:sd} 
and~\ref{thm:esd} are provided, in a unified way,
by Theorem~\ref{thm:uniform}. We recall that a 
triangulation $\Delta$ of a simplicial complex 
$\Gamma$ is called uniform \cite{Ath22} if the 
face vector of the restriction of $\Delta$ to any 
face of $\Gamma$ depends only on the dimension 
of that face. Theorem~\ref{thm:uniform} expresses 
$\ell_V(\Delta,x)$ as a nonnegative linear 
combination of certain polynomials with nonnegative 
coefficients, when $\Delta$ is any uniform 
triangulation of any triangulation $\Gamma$ of the 
simplex $2^V$. The coefficients which play the role 
of the $h$-vector entries $h_k(\Delta)$ in 
Equation~(\ref{eq:BW08}) are sums of local 
$h$-vector entries of restrictions of $\Gamma$ to 
faces of $2^V$. The polynomials 
which play the role of $p_{n,k}(x)$ do not depend 
on $\Gamma$; their basic properties are described 
by Proposition~\ref{prop:lnkj}. In the special case 
of barycentric subdivision, they turn out again to 
be refined descent enumerators of certain permutations 
in $\fS_{n+1}$. These are studied in detail and shown 
to be real-rooted and interlaced by $A_n(x)$ in 
Section~\ref{sec:sd}. The proof is somewhat technical,
but fairly elementary and self-contained. A parallel 
(but less involved) analysis appears for edgewise 
subdivisions in Section~\ref{sec:esd}.  

This paper is organized as follows. 
Section~\ref{sec:pre} contains background on 
real-rooted polynomials and the face enumeration
of simplicial complexes and their triangulations,
including a short overview of uniform 
triangulations. Given a triangulation $\Gamma$ 
of the simplex $2^V$ and a uniform triangulation 
$\Delta$ of $\Gamma$, Section~\ref{sec:uniform} 
proves the formula for the local $h$-polynomial 
$\ell_V(\Delta, x)$ already discussed (see 
Theorem~\ref{thm:uniform}). Sections~\ref{sec:sd} 
and~\ref{sec:esd} apply this formula in the cases 
of barycentric and edgewise subdivisions, 
respectively, and prove Theorems~\ref{thm:sd} 
and~\ref{thm:esd}. For reasons of simplicity, in 
this paper we restrict our attention to geometric 
triangulations of the simplex, rather than deal 
with more general topological simplicial 
subdivisions. However, the main results are valid 
for the class of quasi-geometric simplicial 
subdivisions, introduced in~\cite{Sta92}. Indeed, 
the proof of Theorem~\ref{thm:uniform} works for 
any topological simplicial subdivision $\Gamma$ 
of the simplex and the nonnegativity of the 
coefficients $c_{k,j}(\Gamma)$ which appear there 
(and is exploited in Sections~\ref{sec:sd} 
and~\ref{sec:esd}) is valid for every 
quasi-geometric simplicial subdivision $\Gamma$.

\section{Preliminaries}
\label{sec:pre}

This section includes background from the theory 
of real-rooted polynomials and the face enumeration 
of simplicial complexes and their triangulations 
which will be essential to understand the following
sections; some excellent expositions are \cite{Bj95,
Bra15, StaCCA}. For $n \in \NN$, throughout this 
paper we set $[n] := \{1, 2,\dots,n\}$ and denote 
by $\RR_n [x]$ the space of all polynomials with 
real coefficients of degree at most $n$.

\subsection{Polynomials} 
\label{subsec:polys}

A polynomial $f(x) = a_0 + a_1 x + \cdots + a_n x^n 
\in \RR_n[x]$ is called
\begin{itemize}
\item[$\bullet$] 
  \emph{symmetric}, with center of symmetry $n/2$, 
	if $a_i = a_{n-i}$ for all $0 \le i \le n$,
\item[$\bullet$] 
  \emph{unimodal}, if $a_0 \le a_1 \le \cdots \le 
	              a_k \ge a_{k+1} \ge \cdots \ge a_n$ 
	for some $0 \le k \le n$,
\item[$\bullet$] 
  \emph{$\gamma$-positive}, with center of symmetry 
	$n/2$, if $f(x) = \sum_{i=0}^{\lfloor n/2 \rfloor} 
	\gamma_i x^i (1+x)^{n-2i}$ for some nonnegative real 
	numbers $\gamma_0, \gamma_1,\dots,\gamma_{\lfloor n/2 
	\rfloor}$,
\item[$\bullet$] 
  \emph{real-rooted}, if every root of $f(x)$ is real, 
	or $f(x) \equiv 0$.
\end{itemize}
Every real-rooted, symmetric polynomial with 
nonnegative coefficients is $\gamma$-positive, every 
$\gamma$-positive polynomial is symmetric and unimodal
and every real-rooted polynomial with nonnegative 
coefficients is unimodal; see \cite{Bra15, Ga05} for 
more information on the connections among these 
concepts.

A real-rooted polynomial $f(x)$, with roots 
$\alpha_1 \ge \alpha_2 \ge \cdots$, is said to 
\emph{interlace} a real-rooted polynomial $g(x)$, with 
roots $\beta_1 \ge \beta_2 \ge \cdots$, if
\[ \cdots \le \alpha_2 \le \beta_2 \le \alpha_1 \le
   \beta_1. \]
We then write $f(x) \preceq g(x)$. By convention, 
the zero polynomial interlaces and is interlaced by 
every real-rooted polynomial and nonzero constant 
polynomials interlace all polynomials of 
degree at most one. If two or more real-rooted 
polynomials with positive leading coefficients 
interlace (respectively, are interlaced by) a 
real-rooted polynomial $f(x)$, then so does their 
sum. We will write $\iI_n(f(x)) = x^n f(1/x)$ for 
$f(x) \in \RR_n[x]$. Then, for real-rooted 
polynomials $f(x), g(x) \in \RR_n[x]$ with 
nonnegative coefficients, $f(x) \preceq g(x) 
\Rightarrow \iI_n(g(x)) \preceq \iI_n(f(x))$. We 
will use these properties of interlacing 
frequently in Sections~\ref{sec:sd} 
and~\ref{sec:esd}; see~\cite[Section~7.8]{Bra15} 
for more information.  

A sequence $(f_0(x), f_1(x),\dots,f_n(x))$ of
real-rooted polynomials is called 
\emph{interlacing} if $f_i(x) \preceq f_j(x)$ 
for all $0 \le i < j \le n$. The properties of 
interlacing sequences in the following lemmas 
will be used in Sections~\ref{sec:sd} 
and~\ref{sec:esd}. We recall that, given a 
positive integer $r$, the $i$th Veronese 
$r$-section operator is defined on polynomials
or formal power series by the formula
\[ \sS^r_i \left( \, \sum_{n \ge 0} a_nx^n \right) 
   = \, \sum_{n \ge 0} a_{i+rn} x^n \]
for $i \in \{0, 1,\dots,r-1\}$. We have
\begin{equation} \label{eq:Sx}
\sS^r_i (x^j f(x)) = \begin{cases}
  \sS^r_{i-j} (f(x)), & \text{if $i \ge j$} \\
  x \sS^r_{r-j+i} (f(x)), & \text{if $i < j$} 
\end{cases} \end{equation}
for $i, j \in \{0, 1,\dots,r-1\}$.

\begin{lemma} \label{lem:trans}
{\rm (\cite[Proposition~3.3]{Wa92})}
Let $f_0(x), f_1(x),\dots,f_n(x) \in \RR[x]$ be 
nonzero real-rooted polynomials. If $f_{i-1}(x) 
\preceq f_i(x)$ for every $i \in [n]$ and $f_0(x) 
\preceq f_n(x)$, then $(f_0(x), 
f_1(x),\dots,f_n(x))$ is an interlacing sequence.
\end{lemma} 
\begin{lemma} \label{lem:zhang} 
Let $f(x) \in \RR[x]$ be a polynomial with 
nonnegative coefficients such that 
\[ (\sS^r_{r-j} (f(x)))_{1 \le j \le r} := 
   (\sS^r_{r-1} (f(x)),\dots,\sS^r_1 (f(x)), 
    \sS^r_0 (f(x))) \]
is an interlacing sequence of real-rooted 
polynomials.

\begin{itemize}
\itemsep=0pt
\item[{\rm (a)}]
If $g(x) = x^m f(x)$ for some $m \in \NN$, then 
$(\sS^r_{r-j} (g(x)))_{1 \le j \le r}$ is also an 
interlacing sequence of real-rooted polynomials. 

\item[{\rm (b)}] 
{\rm (cf. \cite[Theorem~2.2]{Zh19})} If $g(x) = 
(1 + x + x^2 + \cdots + x^t) f(x)$ for some $t 
\in \{0, 1,\dots,r-1\}$, then $(\sS^r_{r-j} 
(g(x)))_{1 \le j \le r}$ is also an interlacing 
sequence of real-rooted polynomials.
\end{itemize}
\end{lemma}

\begin{proof}
We only need to prove part (a) and for that, we
may assume that $m=1$. Then, the result is 
equivalent to the obvious fact that $(\sS^r_{r-2} 
(f(x)),\dots,\sS^r_0 (f(x)), x \sS^r_{r-1} 
(f(x)))$ is an interlacing sequence of real-rooted 
polynomials.
\end{proof}

\begin{lemma} \label{lem:new}
Let $f_1(x), f_2(x),\dots,f_n(x) \in \RR[x]$ be 
polynomials with nonnegative coefficients and let 
$g_0(x), g_1(x),\dots,g_n(x)$ be defined by the 
equation 

\begin{equation} \label{eq:new} \begin{pmatrix} 
   g_0(x) \\ g_1(x) \\ \vdots \\ g_n(x)
	 \end{pmatrix} = \begin{pmatrix} 
	 0 & 1 & \cdots & 1 \\ 
   x & 1 & \cdots & 1 \\ 
	 \vdots & \vdots & \ddots & \vdots \\ 
	 x & x & \cdots & 1 \\ x & x & \cdots & x
	\end{pmatrix} \begin{pmatrix} 
  f_1(x) \\ f_2(x) \\ \vdots \\ f_n(x) 
	\end{pmatrix}. \end{equation}
If $(f_1(x), f_2(x),\dots,f_n(x))$ is an 
interlacing sequence of real-rooted polynomials,
then so is the sequence $(g_0(x), 
g_1(x),\dots,g_n(x))$.
\end{lemma} 

\begin{proof}
This follows by a standard application of 
\cite[Theorem~7.8.5]{Bra15}. Indeed, one needs to 
verify that every $2 \times 2$ submatrix of the 
$(n+1) \times n$ matrix on the right-hand side of
Equation~(\ref{eq:new}) preserves interlacing
pairs of polynomials with nonnegative 
coefficients. This has already been 
verified in the literature, for instance in the
proofs of \cite[Corollary~7.8.7]{Bra15} and 
\cite[Theorem~2.2]{Zh19}.
\end{proof}

\subsection{Simplicial complexes and 
            triangulations} 
\label{subsec:complexes}

We assume familiarity with basic notions about 
simplicial complexes and their triangulations 
\cite{Bj95, MW17, StaCCA}. All simplicial 
complexes considered 
here will be finite. We denote by $|V|$ the
cardinality of a finite set $V$ and by $2^V$ the 
collection of all subsets of $V$, which is the 
abstract simplex with vertex set $V$. 

The \emph{$h$-polynomial} of an 
$(n-1)$-dimensional simplicial complex $\Delta$ 
is defined by the formula
\begin{equation}
\label{eq:hdef}
h(\Delta, x) = \sum_{i=0}^n f_{i-1} (\Delta) x^i 
(1-x)^{n-i}, 
\end{equation}
where $f_i(\Delta)$ is the number of $i$-dimensional 
faces of $\Delta$. The sequence $(h_0(\Delta), 
h_1(\Delta),\dots,h_n(\Delta))$ is the 
\emph{$h$-vector} of $\Delta$. 

By the term \emph{triangulation} of a simplicial 
complex $\Gamma$ we will mean a geometric 
triangulation. Thus, a simplicial complex $\Delta$ 
is a triangulation of $\Gamma$ if there exists a 
geometric realization $L$ of $\Delta$ which 
geometrically subdivides a geometric realization 
$K$ of $\Gamma$. The \emph{carrier} of a face 
$G \in \Delta$ is the smallest face 
$F \in \Gamma$ for which the face of $K$ 
corresponding to $F$ contains the face of $L$ 
corresponding to $G$. The \emph{restriction} of 
$\Delta$ to a face $F \in \Gamma$ is the 
subcomplex of $\Delta$ which consists of those 
faces whose carrier is contained in $F$. This 
complex is a triangulation of the simplex $2^F$
and is denoted by $\Delta_F$. A simplicial 
complex $\Delta$ is called \emph{flag} if every
minimal nonface of $\Delta$ has two elements.

The \emph{local $h$-polynomial} of a 
triangulation $\Gamma$ of a simplex $2^V$ is 
defined by Equation~(\ref{eq:localh-def}). By the 
principle of inclusion-exclusion, 
\begin{equation} \label{eq:h-localh}
  h (\Gamma, x) \, = \sum_{F \subseteq V} 
  \ell_F (\Gamma_F, x).
\end{equation}
The polynomial $\ell_V (\Gamma, x)$ is symmetric, 
with center of symmetry $|V|/2$, and has 
nonnegative coefficients (these properties hold,
more generally, for quasi-geometric simplicial 
subdivisions $\Gamma$ of $2^V$); see \cite{Sta92} 
\cite[Section~III.10]{StaCCA} for more information 
about this important concept. 

\medskip
\noindent
\textbf{Barycentric and edgewise subdivisions}. 
Let $\Gamma$ be an $(n-1)$-dimensional simplicial 
complex with vertex set $V(\Gamma)$ and $r$ be a 
positive integer. The \emph{barycentric subdivision} 
of $\Gamma$, denoted by $\sd(\Gamma)$, is 
defined as the simplicial complex of all chains in 
the poset of nonempty faces of $\Gamma$. The edgewise 
subdivision depends on $r$ and a linear ordering of 
$V(\Gamma)$ (although its face vector is independent 
of the latter). Given such an ordering 
$v_1, v_2,\dots,v_m$, we denote by $V_r(\Gamma)$ the 
set of maps $f: V(\Gamma) \to \NN$ such that 
$\supp(f) \in \Gamma$ and $f(v_1) + f(v_2) + \cdots + 
f(v_m) = r$, where $\supp(f)$ is the set of all 
$v \in V(\Gamma)$ for which $f(v) \ne 0$. For $f \in 
V_r(\Gamma)$, we let $\iota(f): V(\Gamma) \to \NN$ be 
the map defined by setting $\iota(f)(v_j) = f(v_1) + 
f(v_2) + \cdots + f(v_j)$ for $j \in \{1, 2,\dots,m\}$. 
The \emph{$r$-fold edgewise subdivision} of $\Gamma$, 
denoted by $\esd_r(\Gamma)$, is the simplicial complex  
on the vertex set $V_r(\Gamma)$ of which a set 
$E \subseteq V_r(\Gamma)$ is a face if the following 
two conditions are satisfied:

\begin{itemize}
\itemsep=0pt
\item[$\bullet$]
$\bigcup_{f \in E} \, \supp(f) \in \Gamma$ and

\item[$\bullet$]
$\iota(f) - \iota(g) \in \{0, 1\}^{V(\Gamma)}$, or 
$\iota(g) - \iota(f) \in \{0, 1\}^{V(\Gamma)}$, for 
all $f, g \in E$.
\end{itemize}

The simplicial complexes $\sd(\Gamma)$ and 
$\esd_r(\Gamma)$ can be realized as triangulations 
of $\Gamma$; see \cite{Ath16b} 
\cite[Section~9.3]{Bj95} \cite[Section~3.2]{MW17}
and the references given there. 

\medskip
\noindent
\textbf{Uniform triangulations}. A triangulation 
$\Delta$ of an $(n-1)$-dimensional simplicial 
complex $\Gamma$ is called \emph{uniform} 
\cite{Ath22} if for all $0 \le i \le j \le n$ and 
for every $(j-1)$-dimensional face $F \in \Gamma$,
the number of $(i-1)$-dimensional faces of the 
restriction $\Delta_F$ depends only on $i$ and $j$.
Denoting this number by $f(i,j)$, we say that the
triangular array $\fF = (f(i,j))_{0 \le i \le j \le 
n}$ is the \emph{$f$-triangle} associated to $\Delta$ 
and that $\Delta$ is an \emph{$\fF$-uniform} 
triangulation of $\Gamma$. Barycentric and $r$-fold
edgewise subdivisions are prototypical examples of
uniform triangulations. Clearly, the restriction of 
an $\fF$-uniform triangulation of $\Gamma$ to any 
subcomplex of $\Gamma$ is an $\fF$-uniform 
triangulation of that subcomplex.

Following~\cite{Ath22}, we denote by $\sigma_n$ 
the $(n-1)$-dimensional abstract simplex. One of the
main results of~\cite{Ath22} shows that for $0 \le 
k \le m \le n$, there exist polynomials 
$p_{\fF,m,k}(x)$ which depend only on $\fF$ and 
$m,k$ such that 
\begin{equation} \label{eq:h-uniform}
h(\Delta,x) = \sum_{k=0}^m h_k(\Gamma)  
                        p_{\fF,m,k}(x) 
\end{equation}
for every $\fF$-uniform triangulation $\Delta$ 
of any $(m-1)$-dimensional simplicial complex 
$\Gamma$. Moreover, the polynomials $p_{\fF,m,k}(x)$ 
have nonnegative coefficients, have the property 
that 
\begin{equation} \label{eq:pnk-recip}
x^m p_{\fF,m,k}(1/x) = p_{\fF,m,m-k}(x)
\end{equation}
for $k \in \{0, 1,\dots,m\}$ and satisfy the 
recurrence 
\begin{equation} \label{eq:pnk-rec}
p_{\fF,m,k}(x) = p_{\fF,m,k-1}(x) + (x-1) 
p_{\fF,m-1,k-1}(x)
\end{equation}
for $1 \le k \le m \le n$; see 
\cite[Section~6]{Ath22}. As a result of 
Equations~(\ref{eq:h-uniform}) 
and~(\ref{eq:pnk-rec}) we have $p_{\fF,m,0}(x) = 
h_\fF(\sigma_m,x)$, where $h_\fF(\sigma_m,x)$ 
stands for the $h$-polynomial of any 
$\fF$-uniform triangulation of $\sigma_m$, and  
\begin{equation} \label{eq:pnk}
p_{\fF,m,k}(x) = \sum_{i=0}^k {k \choose i} 
                 (x-1)^i p_{\fF,m-i,0}(x)
							 = \sum_{i=0}^k {k \choose i} 
                 (x-1)^i h_\fF(\sigma_{m-i},x)
\end{equation}
for every $k \in \{0, 1,\dots,m\}$.

Following~\cite{Ath22} again, for $m \in 
\{0, 1,\dots,n\}$ we denote by $\ell_\fF(\sigma_m,x)$ 
the local $h$-polynomial of any $\fF$-uniform 
triangulation of $\sigma_m$. The polynomials
\begin{equation} \label{eq:lnk-def}
\ell_{\fF,m,k}(x) = \sum_{i=0}^k (-1)^i {k \choose i}
                      h_\fF(\sigma_{m-i},x)
									= \sum_{i=0}^k (-1)^i {k \choose i}
                      p_{\fF,m-i,0}(x),
\end{equation}
introduced for $k \in \{0, 1,\dots,m\}$
in~\cite[Section~5.3]{Ath23+}, play an important role
in this paper. They interpolate between $\ell_{\fF,m,0}
(x) = p_{\fF,m,0}(x) = h_\fF(\sigma_m,x)$ and 
$\ell_{\fF,m,m}(x) = \ell_\fF(\sigma_m,x)$. They have
nonnegative coefficients, since 
\begin{equation} \label{eq:lnk-alt}
\ell_{\fF,m,k}(x) = \sum_{i=0}^{m-k} {m-k \choose i} 
                    \ell_\fF(\sigma_{m-i},x)
\end{equation}
by \cite[Remark~6.2~(c)]{Ath23+}.

\section{Uniform triangulations}
\label{sec:uniform}

This section considers $\fF$-uniform triangulations, 
where $\fF$ is an arbitrary $f$-triangle of size $n$.
For $m \in \{0, 1,\dots,n\}$, we introduce the 
polynomials 
\begin{equation} \label{eq:lnkj-def}
\ell_{\fF,m,k,j}(x) = \sum_{i=0}^k (-1)^i {k \choose i}
                      p_{\fF,m-i,j}(x)
\end{equation}
for $k, j \in \{0, 1,\dots,m\}$ with $k + j \le m$.
The following statement is the main result of this 
section.
\begin{theorem} \label{thm:uniform} 
Let $\Gamma$ be any triangulation of the 
$(n-1)$-dimensional simplex $2^V$. Then, there exist 
nonnegative integers $c_{k,j}(\Gamma)$ for $k, j \in 
\{0, 1,\dots,n\}$ with $k + j \le n$ such that 
\begin{equation} \label{eq:local-uniform}
\ell_V(\Delta,x) = \sum_{k=0}^n 
\sum_{j=0}^{n-k} c_{k,j}(\Gamma) \ell_{\fF,n,k,j}(x)
\end{equation}
for every $\fF$-uniform triangulation $\Delta$ of 
$\Gamma$. Specifically,
\begin{equation} \label{eq:ckj}
c_{k,j}(\Gamma) = \sum_{F \subseteq V: \, |F|=n-k} 
                  [x^j] \ell_F(\Gamma_F,x) 
\end{equation}
for $k, j \in \{0, 1,\dots,n\}$ with $k + j \le n$.
\end{theorem}

\begin{proof}
Since $\Delta_G$ is an $\fF$-uniform triangulation
of $\Gamma_G$ for every $G \subseteq V$, applying 
Equations~(\ref{eq:h-localh}) 
and~(\ref{eq:h-uniform}) we get 

\begin{align*}
\ell_V(\Delta,x) & = \sum_{G \subseteq V} 
     (-1)^{n - |G|} \, h (\Delta_G, x) \ = 
		 \sum_{G \subseteq V} (-1)^{n - |G|} 
     \sum_{j=0}^{|G|} h_j(\Gamma_G) p_{\fF,|G|,j}(x) 
		\\ & = \sum_{0 \le j \le i \le n} (-1)^{n-i} 
		       p_{\fF,i,j}(x)
    \sum_{G \subseteq V: \, |G|=i} h_j(\Gamma_G).  
\end{align*}

\medskip
\noindent
We now use Equation~(\ref{eq:h-localh}) to express 
$h_j(\Gamma_G)$ as a sum of local $h$-vector entries 
of restrictions of $\Gamma$ to faces $F \subseteq G$
of $2^V$ and change the order of summation to get 

\begin{align*}
\ell_V(\Delta,x) & = 
\sum_{0 \le j \le i \le n} (-1)^{n-i} p_{\fF,i,j}(x)
\sum_{F \subseteq G \subseteq V: \, |G|=i} [x^j] 
      \ell_F(\Gamma_F,x) \\
			& = \sum_{F \subseteq V} \sum_{j=0}^{|F|} \,
			[x^j] \ell_F(\Gamma_F,x) 
			\sum_{i=j}^n \, (-1)^{n-i} 
			{n-|F| \choose n-i} p_{\fF,i,j}(x).
\end{align*}

\medskip
\noindent
Setting $|F| = n-k$, the resulting expression can be
rewritten as 

\begin{align*}
\ell_V(\Delta,x) & = 
\sum_{0 \le j \le n-k \le n} \,
\sum_{F \subseteq V: \, |F|=n-k} [x^j] 
                                 \ell_F(\Gamma_F,x) 
			\sum_{i=n-k}^n \, (-1)^{n-i} 
			{k \choose n-i} p_{\fF,i,j}(x) \\ & = 
\sum_{0 \le j \le n-k \le n} \left( \,
\sum_{F \subseteq V: \, |F|=n-k} [x^j] 
                     \ell_F(\Gamma_F,x) \right) 
			\sum_{i=0}^k \, (-1)^i {k \choose i} 
			p_{\fF,n-i,j}(x) \\ & = 
			\sum_{0 \le j \le n-k \le n} \left( \,
\sum_{F \subseteq V: \, |F|=n-k} [x^j] 
        \ell_F(\Gamma_F,x) \right) 
				\ell_{\fF,n,k,j}(x)
\end{align*}

\medskip
\noindent
and the proof follows.
\end{proof}

The following statement collects the main properties 
of the polynomials $\ell_{\fF,m,k,j}(x)$.
\begin{proposition} \label{prop:lnkj} 
Let $m \in \NN$ with $m \le n$.
\begin{itemize}
\itemsep=0pt
\item[(a)] 
The polynomial $\ell_{\fF,m,k,j}(x)$ has 
nonnegative integer coefficients for all $k, j 
\in \{0, 1,\dots,m\}$ with $k + j \le m$.

\item[(b)] 
We have $x^m \ell_{\fF,m,k,j}(1/x) = 
\ell_{\fF,m,k,m-k-j}(x)$ for all $k, j \in 
\{0, 1,\dots,m\}$ with $k + j \le m$.

\item[(c)] 
We have $\ell_{\fF,m,k,0}(x) = \ell_{\fF,m,k}(x)$ 
for every $k \in \{0, 1,\dots,m\}$. 

\item[(d)] 
We have $\ell_{\fF,m,0,j}(x) = p_{\fF,m,j}(x)$ 
for every $j \in \{0, 1,\dots,m\}$.

\item[(e)] 
The recurrence 
\begin{equation*} %%\label{eq:lnkj-rec1}
\ell_{\fF,m,k,j}(x) = \ell_{\fF,m,k,j-1}(x) + (x-1) 
\ell_{\fF,m-1,k,j-1}(x)
\end{equation*}
holds for all $k \in \{0, 1,\dots,m\}$ and $j \in 
\{1, 2,\dots,m-k\}$.

\item[(f)] 
The recurrence 
\begin{equation*} %%\label{eq:lnkj-rec2}
\ell_{\fF,m,k,j}(x) = \ell_{\fF,m,k-1,j}(x) -
\ell_{\fF,m-1,k-1,j}(x)
\end{equation*}
holds for all $j \in \{0, 1,\dots,m\}$ and $k \in 
\{1, 2,\dots,m-j\}$.

\end{itemize}
\end{proposition}

\begin{proof}
Parts (c) - (f) are consequences of the definition 
of $\ell_{\fF,m,k,j}(x)$, Equation~(\ref{eq:lnk-def})
and the recurrence~(\ref{eq:pnk-rec}). For part (a) 
we apply induction on $m$ and $j$. For $j=0$ the 
result holds by part (c). For $j \ge 1$, we employ 
parts (e) and (f) to find that
\begin{align*} \ell_{\fF,m,k,j}(x) & = x 
\ell_{\fF,m-1,k,j-1}(x) + \left( \ell_{\fF,m,k,j-1}
(x) - \ell_{\fF,m-1,k,j-1}(x) \right) \\ 
& = x \ell_{\fF,m-1,k,j-1}(x) + 
      \ell_{\fF,m,k+1,j-1}(x) 
\end{align*}
and apply induction. Similarly, for part (b) we 
apply induction on $m$ and $k$. For $k=0$ the result 
holds by part (d) and Equation~(\ref{eq:pnk-recip}). 
For $k \ge 1$, we apply parts (e) and (f) and the 
induction hypothesis to find that
\begin{align*} x^m \ell_{\fF,m,k,j}(1/x) & = 
x^m \ell_{\fF,m,k-1,j}(1/x) - x^m 
    \ell_{\fF,m-1,k-1,j}(1/x) \\ & = 
\ell_{\fF,m,k-1,m-k-j+1}(x) - x 
    \ell_{\fF,m-1,k-1,m-k-j}(x) \\ & = 
\ell_{\fF,m,k-1,m-k-j}(x) -
    \ell_{\fF,m-1,k-1,m-k-j}(x) \\ & =
\ell_{\fF,m,k,m-k-j}(x).
\end{align*}
\end{proof}

\section{Barycentric subdivisions}
\label{sec:sd}

Throughout this section, $\fF$ stands for the 
$f$-triangle of size $n$ for barycentric subdivision.
The goal is to prove the following result and deduce 
from it Theorem~\ref{thm:sd}.
\begin{theorem} \label{thm:sd-interlace} 
The sequence $(\ell_{\fF,n,k,j}(x))_{0 \le j \le n-k}$ 
is an interlacing sequence of real-rooted polynomials
for all $0 \le k \le n$. Moreover, each term of this 
sequence is interlaced by the Eulerian polynomial 
$A_n(x)$. 
\end{theorem}

Throughout this section, we write $p_{n,k}(x)$ 
instead of $p_{\fF,n,k}(x)$ for $0 \le k \le n$. 
To describe these polynomials explicitly, we need to
recall some notation and terminology about the 
combinatorics of permutations; our basic reference 
is~\cite{StaEC1}. We denote by $\fS_n$ the symmetric 
group of permutations of the set $[n]$. An index
$i \in [n-1]$ is called a \emph{descent} 
(respectively, \emph{ascent} or \emph{excedance}) of 
a permutation $w \in \fS_n$ if $w(i) > w(i+1)$ 
(respectively, $w(i) < w(i+1)$ or $w(i) > i$). The 
number of descents (respectively, ascents or
excedances) of $w$ is denoted by $\des(w)$ 
(respectively, by $\asc(w)$ or $\exc(w)$). We then 
have (see \cite[Example~3.8]{Ath23+}) 
$p_{n,k}(x) = 1$ for $n=k=0$ and 
\begin{align} 
p_{n,k}(x) & = \sum_{w \in \fS_{n+1} : \, w(1) = k+1} 
                  x^{\des(w)} = 
							 \sum_{w \in \fS_{n+1} : \, w(n+1) = k+1} 
                  x^{\asc(w)} 	\label{eq:pnk-1} \\
           & = \sum_{w \in \fS_{n+1}: \, w^{-1}(1) = 
					           k+1} x^{\exc(w)} \label{eq:pnk-2}
\end{align}
for $n \ge 1$ and $k \in \{0, 1,\dots,n\}$. In 
particular, $p_{n,0}(x) = A_n(x)$ and $p_{n,n}(x) 
= xA_n(x)$ for $n \ge 1$, where  
\[ A_n(x) \, = \sum_{w \in \fS_n} x^{\des(w)} 
          \, = \sum_{w \in \fS_n} x^{\asc(w)} 
					\, = \sum_{w \in \fS_n} x^{\exc(w)} \]
is the $n$th \emph{Eulerian polynomial}.

{\scriptsize
\begin{table}[hptb]
\begin{center}
\begin{tabular}{| l || l | l | l | l | l | l ||} 
       \hline
& $k=0$ & $k=1$ & $k=2$ & $k=3$ & $k=4$ \\ \hline 
          \hline
$n=0$  & 1 &  & & & \\ \hline
$n=1$  & 1 & 0 & & & \\ \hline
$n=2$  & $1+x$ & $x$ & $x$ & & \\ \hline
$n=3$  & $1+4x+x^2$ & $3x+x^2$ & $2x+x^2$ & $x+x^2$ 
       & \\ \hline
$n=4$  & $1+11x+11x^2+x^3$ & $7x+10x^2+x^3$ & 
         $4x+9x^2+x^3$ & $2x+8x^2+x^3$ & $x+7x^2+x^3$ 
			 \\ \hline
\end{tabular}

\bigskip
\caption{The polynomials $d_{n,k}(x)$ for $n \le 4$.}
\label{tab:dnk}
\end{center}
\end{table}}

We denote by $\Fix(w)$ the set of fixed points of 
a permutation $w$ and by $\dD_{n,k}$ the set of 
permutations $w \in \fS_n$ such that $\Fix(w) 
\subseteq [n-k]$. The polynomials (see 
Table~\ref{tab:dnk})
\begin{equation} \label{eq:def-dnk}
d_{n,k}(x) = \sum_{w \in \dD_{n,k}} x^{\exc(w)},
\end{equation}
where $n \ge 1$ and $k \in \{0, 1,\dots,n\}$, 
were studied implicitly in 
\cite[Section~3.2]{BS21} and explicitly in 
\cite[Section~4]{Ath23+} (we also set $d_{n,k}
(x) := 1$ for $n=k=0$). These polynomials 
interpolate between the Eulerian polynomial 
$d_{n,0}(x) = A_n(x)$ and 
\[ d_{n,n}(x) = \sum_{w \in \fS_n: \, \Fix(w) = 
   \varnothing} x^{\exc(w)}, \]
called the $n$th \emph{derangement polynomial};
see, for instance, \cite[Section~3.2]{BS21} 
\cite{GS20, HZ19} and references therein. As
part of their study of the derangement 
transformation \cite[Section~3.2]{BS21}, using 
tools from the theory of multivariate stable 
polynomials, Br\"and\'en and Solus showed that 
$(d_{n,k}(x))_{0 \le k \le n}$ is an interlacing 
sequence of real-rooted polynomials for every 
$n \in \NN$ (this follows from 
\cite[Theorem~3.6]{BS21}). Our proof of 
Theorem~\ref{thm:sd-interlace} yields a more 
elementary proof of this fact. We need to focus 
on further refined polynomials, defined as
$d_{n,k,j}(x) := 1$ for $n=k=j=0$ and
\begin{equation} \label{eq:def-dnkj}
d_{n,k,j}(x) = \sum_{w \in \dD_{n+1,k}: \, 
               w^{-1}(1) = j+1} x^{\exc(w)} 
\end{equation}
for $n \ge 1$ and $k, j \in \{0, 1,\dots,n\}$;
see Tables~\ref{tab:dn1j} and~\ref{tab:dn2j} 
for small values of $n, k$. 
We will show that $\ell_{\fF,n,k,j}(x) = 
d_{n,k,j}(x)$ in the range $k + j \le n$ of 
interest. We note that $d_{n,k,0}(x) = d_{n,k}
(x)$ for all $n,k$ and that the polynomials 
$d_{n,n,j}(x)$ have appeared, in different  
notation, on \cite[p.~46]{HZ19}.

{\scriptsize
\begin{table}[hptb]
\begin{center}
\begin{tabular}{| l || l | l | l | l | l | l ||} 
       \hline
& $j=0$ & $j=1$ & $j=2$ & $j=3$ & $j=4$ \\ \hline 
          \hline
$n=1$  & 0 & $x$ & & & \\ \hline
$n=2$  & $x$ & $x$ & $x+x^2$ & & \\ \hline
$n=3$  & $3x+x^2$ & $2x+2x^2$ & $x+3x^2$ & 
         $x+4x^2+x^3$ & \\ \hline
$n=4$  & $7x+10x^2+x^3$ & $4x+12x^2+2x^3$ & 
         $2x+12x^2+4x^3$ & $x+10x^2+7x^3$ & 
				 $x+11x^2+11x^3+x^4$ \\ \hline
\end{tabular}

\bigskip
\caption{The polynomials $d_{n,k,j}(x)$ for $k=1$ 
         and $n \le 4$.}
\label{tab:dn1j}
\end{center}
\end{table}}

{\scriptsize
\begin{table}[hptb]
\begin{center}
\begin{tabular}{| l || l | l | l | l | l | l ||} 
       \hline
& $j=0$ & $j=1$ & $j=2$ & $j=3$ & $j=4$ \\ \hline 
          \hline
$n=2$  & $x$ & $x$ & $x^2$ & & \\ \hline
$n=3$  & $2x+x^2$ & $x+2x^2$ & $x+3x^2$ & $3x^2+x^3$ 
       & \\ \hline
$n=4$  & $4x+9x^2+x^3$ & $2x+10x^2+2x^3$ & 
         $x+9x^2+4x^3$ & $x+10x^2+7x^3$ & 
				 $7x^2+10x^3+x^4$ \\ \hline
\end{tabular}

\bigskip
\caption{The polynomials $d_{n,k,j}(x)$ for $k=2$ 
         and $n \le 4$.}
\label{tab:dn2j}
\end{center}
\end{table}}

Equation~(\ref{eq:pnk-2}) is a consequence of  
the first fundamental transformation 
\cite[Section~I.3]{FSc70} 
\cite[Section~I.3]{StaEC1}, defined when each 
cycle of a permutation $w \in \fS_{n+1}$ is 
written with its smallest element first and 
cycles are arranged in decreasing order of 
their smallest element. We say that $w(j) \in 
[n]$ is a \emph{bad point} of $w \in \fS_n$ if 
$j$ is a left-to-right minimum of $w$ (meaning 
that $w(i) \ge w(j)$ for $i \le j$) and either 
$i = n$ or $w(i) > w(i+1)$ (this notion 
corresponds to the one adopted in 
\cite[Section~3]{Ath23+} after reversing the 
permutation). The first fundamental transformation 
is a bijection $\varphi_n : \fS_n \to \fS_n$ such 
that for every $w \in \fS_n$ (a) $\exc(w) = 
\asc(\varphi(w))$; (b) $\Fix(w)$ is equal to the 
set of bad points of $\varphi(w)$; and (c) $w^{-1}
(1) = \varphi(w)(n)$. As a consequence, 
\begin{equation} \label{eq:dnk-bad}
d_{n,k}(x) = \sum_{w \in \bB_{n,k}} x^{\asc(w)}, 
\end{equation}
where $\bB_{n,k}$ stands for the set of permutations 
$w \in \fS_n$ such that the set of bad points of $w$ 
is contained in $[n-k]$.

The following two propositions collect the 
properties of the polynomials $d_{n,k,j}(x)$ which
will be needed in this section. We note that part 
(b) of the next proposition fails to hold outside 
the range $j+k \le n$, where $\ell_{\fF,n,k,j}(x)$
is defined by Equation~(\ref{eq:lnkj-def}) and 
$p_{\fF,n,k}(x) := 0$ for $n<k$.
\begin{proposition} \label{prop:dnkj} 
Let $n \in \NN$.
\begin{itemize}
\itemsep=0pt
\item[(a)] 
We have $\ell_{\fF,n,k}(x) = d_{n,k}(x)$ for every 
$k \in \{0, 1,\dots,n\}$.

\item[(b)] 
We have $\ell_{\fF,n,k,j}(x) = d_{n,k,j}(x)$ for all 
$k, j \in \{0, 1,\dots,n\}$ with $k + j \le n$. In 
particular, Proposition~\ref{prop:lnkj} applies to 
the polynomials $d_{n,k,j}(x)$ with $0 \le j \le n-k$.

\item[(c)] 
We have 
\[ d_{n,k,j}(x) = \sum_{w \in \bB_{n+1,k}: \, 
                  w(n+1) = j+1} x^{\asc(w)} \]
for all $k, j \in \{0, 1,\dots,n\}$.

\item[(d)] 
For $n \ge 1$, we have
\[ d_{n,k,0}(x) = d_{n,k}(x) = \sum_{j=0}^{n-1} 
                  d_{n-1,k,j}(x) \]
for every $k \in \{0, 1,\dots,n-1\}$.

\item[(e)] 
For $n \ge 1$ and $k \in \{1, 2,\dots,n\}$, we have 
\[ d_{n,k,j}(x) = \begin{cases}
   d_{n,k-1,j}(x) - d_{n-1,k-1,j}(x), & 
	        \text{if $0 \le j \le n-k$} \\
   d_{n,k-1,n-k+1}(x), & \text{if $j = n-k+1$} \\
	 d_{n,k-1,j}(x) - d_{n-1,k-1,j-1}(x), & 
	        \text{if $n-k+1 < j \le n$.} \\
                  \end{cases} \]

\item[(f)] 
For $n \ge k \ge 1$, we have
\[ d_{n,k,j}(x) = x \sum_{i=0}^{j-1} d_{n-1,k-1,i}(x) 
                  + \sum_{i=j}^{n-1} d_{n-1,k-1,i}(x) 
									  \]
for $n-k+1 \le j \le n$.

\item[(g)]
We have $d_{n,k,n}(x) = x d_{n,k-1}(x)$ for $1 \le k 
\le n$.

\end{itemize}
\end{proposition}

\begin{proof}
By combining Equations~(\ref{eq:lnkj-def}) 
and~(\ref{eq:pnk-2}) we get 
\[ \ell_{\fF,n,k,j}(x) = 
   \sum_{i=0}^k (-1)^i {k \choose i}
   \sum_{w \in \fS_{n-i+1}: \, w^{-1}(1) = 
		  			           j+1} x^{\exc(w)}. \]
Since $0 \le j \le n-k$, part (b) follows from this 
expression via a standard application of the 
principle of inclusion-exclusion. Part (a) is the 
special case $j=0$ of part (b) (and is equivalent 
to the result of \cite[Example~6.3~(a)]{Ath23+} as
well). Part (c) follows from 
Equation~(\ref{eq:def-dnkj}) via Foata's
fundamental transformation, already discussed.
Part (d) follows from Equation~(\ref{eq:def-dnk})
by conditioning on the value of $w^{-1}(1)$ for $w 
\in \dD_{n,k}$. 

To verify part (e) we apply the definition of 
$d_{n,k,j}(x)$ and note that
\[ d_{n,k-1,j}(x) - d_{n,k,j}(x) = \sum
   x^{\exc(w)}, \]
where the sum ranges over all permutations $w \in 
\fS_{n+1}$ such that $w^{-1}(1) = j+1$, $\Fix(w) 
\subseteq [n-k+2]$ and $n-k+2 \in \Fix(w)$.
Clearly, this sum is equal to $d_{n-1,k-1,j}(x)$
if $0 \le j \le n-k$, to zero if $j = n-k+1$, and 
to $d_{n-1,k-1,j-1}(x)$ if $n-k+1 < j \le n$.

Part (f) follows from the interpretation of 
$d_{n,k,j}(x)$ given in part (c) by conditioning 
on the possible values of $w(n)$. Part (g) follows
from parts (d) and (f).
\end{proof}

Part (b) of the following proposition is due to
Haglund and Zhang \cite[p.~46]{HZ19}.
\begin{proposition} \label{prop:dnkj-rec} 
Let $n \ge 1$.
\begin{itemize}
\itemsep=0pt
\item[(a)]  
For every $k \in \{0, 1,\dots,n-1\}$ we have
\[ d_{n,k,j}(x) = x \sum_{i=0}^{j-1} d_{n-1,k,i}(x) 
                  + \sum_{i=j}^{n-1} d_{n-1,k,i}(x) 
									  \]
for $0 \le j \le n-k$ and 
\[ d_{n,k,j}(x) = x \sum_{i=0}^{j-2} d_{n-1,k,i}(x) 
                  + (1+x) d_{n-1,k,j-1}(x) + 
									  \sum_{i=j}^{n-1} d_{n-1,k,i}(x) 
									  \]
for $n-k+1 \le j \le n$.

\item[(b)] {\rm (cf. \cite[p.~46]{HZ19})}
We have 
\[ d_{n,n,0}(x) = \sum_{i=1}^{n-1} d_{n-1,n-1,i}(x) 
   \]
and 
\[ d_{n,n,j}(x) = x \sum_{i=0}^{j-1} d_{n-1,n-1,i}(x) 
                  + \sum_{i=j}^{n-1} d_{n-1,n-1,i}(x)
									\]
for $1 \le j \le n$.
\end{itemize}
\end{proposition}

\begin{proof}
Let $k \in \{0, 1,\dots,n-1\}$. Given part (d) of 
Proposition~\ref{prop:dnkj}, the first formula of 
part (a) follows by induction on $j$ from the 
recurrence 
\[ d_{n,k,j}(x) = d_{n,k,j-1}(x) + (x-1) 
   d_{n-1,k,j-1}(x), \]

\medskip
\noindent
which holds for $1 \le j \le n-k$ by 
Proposition~\ref{prop:lnkj} (e) and 
Proposition~\ref{prop:dnkj} (b). A comparison of 
the second formula to that of 
Proposition~\ref{prop:dnkj} (f) shows that the 
former is equivalent to 
\[ d_{n-1,k,j-1}(x) = x \sum_{i=0}^{j-1} 
   \left( d_{n-1,k-1,i}(x) - d_{n-1,k,i}(x) 
	            \right) + \sum_{i=j}^{n-1}
	 \left( d_{n-1,k-1,i}(x) - d_{n-1,k,i}(x) 
	            \right). \]
In view of part (e) of Proposition~\ref{prop:dnkj},
and given that $j \ge n-k+1$, the latter equality 
can be rewritten as 
\[ d_{n-1,k,j-1}(x) = x \sum_{i=0}^{j-2} 
   d_{n-2,k-1,i}(x) + \sum_{i=j-1}^{n-2} 
	                             d_{n-2,k-1,i}(x). \]
This holds by Proposition~\ref{prop:dnkj} (f) and
the proof of part (a) follows. Part (b) follows 
from the interpretation of $d_{n,n,j}(x)$ given 
in Proposition~\ref{prop:dnkj} (c) by conditioning 
on the possible values of $w(n)$. 
\end{proof}

\begin{theorem} \label{thm:dnkj} 
The sequence $(d_{n,k,j}(x))_{0 \le j \le n}$ is 
an interlacing sequence of real-rooted polynomials
for all $0 \le k \le n$. Moreover, $d_{n,k,j}(x)$ 
is interlaced by the Eulerian polynomial $A_n(x)$
for every $j \in \{0, 1,\dots,n-k\}$. 
\end{theorem}

\begin{proof}
The result is trivial for $n \in \{0, 1\}$, so 
we assume that $n \ge 2$. By 
Proposition~\ref{prop:dnkj-rec} we have 

\[ \begin{pmatrix} 
   d_{n,k,0}(x) \\ \vdots \\ d_{n,k,n-k}(x) \\ 
	 d_{n,k,n-k+1}(x) \\ \vdots \\ d_{n,k,n}(x) 
	 \end{pmatrix} = 
   \begin{pmatrix} 
	 1 & 1 & \cdots & 1 & \cdots & 1 \\ 
   x & 1 & \cdots & 1 & \cdots & 1 \\ 
	 \vdots & \vdots & \ddots & \vdots & \cdots & 
	 \vdots \\ 
	 x & x & \cdots & 1 & \cdots & 1 \\ 
	 x & x & \cdots & 1+x & \ddots & 1 \\ 
	 \vdots & \vdots & \cdots & \vdots & \ddots &
	 \vdots \\ x & x & \cdots & x & \cdots & 1+x
	\end{pmatrix}
  \begin{pmatrix} 
  d_{n-1,k,0}(x) \\ d_{n-1,k,1}(x) \\ \vdots 
	\\ d_{n-1,k,n-1}(x) 
	\end{pmatrix} \]

\medskip
\noindent
for $k \in \{0, 1,\dots,n-1\}$ and 

\[ \begin{pmatrix} 
   d_{n,k,0}(x) \\ d_{n,k,1}(x) \\ \vdots \\ 
   d_{n,k,n-1}(x) \\ d_{n,k,n}(x) 
	 \end{pmatrix} = 
   \begin{pmatrix} 
	 0 & 1 & \cdots & 1 \\ 
   x & 1 & \cdots & 1 \\ 
	 \vdots & \vdots & \ddots & \vdots \\ 
	 x & x & \cdots & 1 \\ x & x & \cdots & x
	\end{pmatrix}
  \begin{pmatrix} 
  d_{n-1,k-1,0}(x) \\ d_{n-1,k-1,1}(x) \\ \vdots 
	\\ d_{n-1,k-1,n-1}(x) 
	\end{pmatrix} \]

\medskip
\noindent
for $k=n$. In both cases, multiplication on the 
left by the given matrix preserves the 
interlacing of sequences of real-rooted 
polynomials with nonnegative coefficients (see 
\cite[Theorem~2.4]{HZ19} for a more general 
statement, in the former case, and 
Lemma~\ref{lem:new} in the latter) and thus the 
first statement follows by induction on $n$.

We now deduce that $A_n(x) \preceq d_{n,k,j}(x)$
for every $n \in \NN$ and all $0 \le k, j \le n$
with $k+j \le n$. We first show that the sequence 
$(d_{n,k}(x))_{0 \le k \le n}$ is interlacing for 
every $n \in \NN$ (as already mentioned, this also 
follows from \cite[Theorem~3.6]{BS21}). Indeed, by 
the first statement and Proposition~\ref{prop:dnkj} 
(g) we have 
\[ d_{n,k}(x) = d_{n,k,0}(x) \preceq d_{n,k,n}(x)
   = x d_{n,k-1}(x) \]
for $1 \le k \le n$ and hence $d_{n,k-1}(x) \preceq 
d_{n,k}(x)$. Since $d_{n,n}(x)$ has symmetric 
coefficients, with center of symmetry $n/2$, we 
conclude that 
\begin{align*} 
A_n(x) & = d_{n,0}(x) \preceq d_{n,1}(x) \preceq 
\cdots \preceq d_{n,n}(x) \\ & = \iI_n (d_{n,n}(x)) 
\preceq \cdots \preceq \iI_n (d_{n,1}(x)) \preceq 
\iI_n (d_{n,0}(x)) \\ 
       & = \iI_n(A_n(x)) = x A_n(x).
\end{align*}
Since $A_n(x) \preceq x A_n(x)$, our claim follows 
by an application of Lemma~\ref{lem:trans}. We now 
recall from Proposition~\ref{prop:lnkj} (b) that 
$d_{n,k,n-k-j}(x) = \iI_n (d_{n,k,j}(x))$ for all $0 
\le j \le n-k$ and conclude that 
\begin{align*} 
A_n(x) & \preceq d_{n,k}(x) = d_{n,k,0}(x) \preceq 
d_{n,k,1}(x) \preceq \cdots \preceq d_{n,k,n-k}(x) 
\\ & = \iI_n (d_{n,k,0}(x)) = \iI_n (d_{n,k}(x))
       \preceq \iI_n(A_n(x)) \\ & = x A_n(x) 
\end{align*}
for every $k \in \{0, 1,\dots,n\}$. Another 
application of Lemma~\ref{lem:trans} shows that
$A_n(x) \preceq d_{n,k,j}(x)$ for all $0 \le j 
\le n-k$ and the proof follows.
\end{proof}

\begin{proof}[Proof of 
              Theorem~\ref{thm:sd-interlace}.]
Since $\ell_{\fF,n,k,j}(x) = d_{n,k,j}(x)$ for 
$0 \le j \le n-k$ by Proposition~\ref{prop:dnkj} 
(b), the result follows from 
Theorem~\ref{thm:dnkj}.
\end{proof}

\begin{proof}[Proof of Theorem~\ref{thm:sd}.]
Theorem~\ref{thm:uniform} shows that $\ell_V
(\sd(\Gamma),x)$ can be expressed as a nonnegative
linear combination of the polynomials 
$\ell_{\fF,n,k,j}(x)$ with $k, j \in \{0, 
1,\dots,n\}$ and $k + j \le n$. Since each of these 
polynomials is real-rooted and interlaced by 
$A_n(x)$ by Theorem~\ref{thm:sd-interlace}, so 
does $\ell_V(\sd(\Gamma),x)$.
\end{proof}

We close this section with the following corollary, 
which gives a partial answer to 
\cite[Problem~4.12]{Ath16a}.
\begin{corollary} \label{cor:2sd} 
The local $h$-polynomial of the second barycentric 
subdivision $\sd^{(2)}(2^V)$ of the 
$(n-1)$-dimensional simplex is given by the formula
\begin{equation} \label{eq:2sd}
\ell_V(\sd^{(2)}(2^V),x) = \sum_{k=0}^n {n \choose k} 
\sum_{j=0}^{n-k} D_{n-k,j} \, d_{n,k,j}(x)
\end{equation}
for every $n \in \NN$, where $D_{n,j}$ denotes the 
number of permutations $w \in \fS_n$ with no fixed 
points and $j$ excedances.

Equivalently, the coefficient of $x^i$ in the 
polynomial $\ell_V(\sd^{(2)}(2^V),x)$ is equal to 
the number of pairs $(u,v) \in \fS_n \times 
\fS_{n+1}$ such that $\exc(v) = i$, $v^{-1}(1) = 
\exc(u) + 1$ and $v$ has no fixed point larger than 
$n - \fix(u) + 1$, where $\fix(u)$ is the number of 
fixed points of $u$.
\end{corollary}

\begin{proof}
We apply Theorem~\ref{thm:uniform} for $\Gamma = 
\sd(2^V)$ and $\Delta = \sd(\Gamma)$. We have 
$\ell_{\fF,n,k,j}(x) = d_{n,k,j}(x)$ by  
Proposition~\ref{prop:dnkj}~(b) and 
$\ell_F(\Gamma_F,x) = d_{|F|,|F|}(x)$ for every 
$F \subseteq V$, since $\Gamma_F$ is the barycentric 
subdivision of $2^F$. Since there are ${n \choose k}$ 
sets $F \subseteq V$ of cardinality $n-k$, the 
proof of the first statement follows. The second 
statement follows directly from the first.
\end{proof}

\section{Edgewise subdivisions}
\label{sec:esd}

This section proves Theorem~\ref{thm:esd}, following 
the strategy of the proof of Theorem~\ref{thm:sd}, 
given in Section~\ref{sec:sd}, and using the special 
features of the combinatorics of edgewise subdivisions.
Throughout it, $\fF$ stands for the $f$-triangle of 
size $n$ for the $r$-fold edgewise subdivision for 
some fixed integer $r \ge n$. 

We let $E_{n,r}(x) = p_{\fF,n,0}(x) = h_\fF(\sigma_n,
x)$ be the $h$-polynomial of the $r$-fold edgewise 
subdivision of the $(n-1)$-dimensional simplex. By 
\cite[Equation~(8)]{Ath16b} we have
\begin{equation} 
E_{n,r}(x) = \sS^r_0 \left( (1 + x + \cdots 
                             + x^{r-1})^n \right) 
           = \sum_{w \in \wW_{n,r}} x^{\asc(w)},
\label{eq:Enr} 
\end{equation}
where $\wW_{n,r}$ is the set of words $w: \{0, 
1,\dots,n-1\} \to \{0, 1,\dots,r-1\}$ with $w(0) = 0$ 
and $\asc(w)$ is the number of indices $i \in [n-1]$ 
such that $w(i-1) < w(i)$.

\begin{proposition} \label{prop:esdr} 
The following formulas hold for the $r$-fold 
edgewise subdivision for $n \in \NN$ and $k, j \in
\{0, 1,\dots,n\}$:
\begin{align*} 
p_{\fF,n,k}(x) & = \sS^r_0 
   \left( x^k (1 + x + \cdots + x^{r-1})^n 
	 \right), \\ 
\ell_{\fF,n,k}(x) & = \sS^r_0 \left( 
      (1 + x + \cdots + x^{r-1})^{n-k} 
			(x + x^2 + \cdots + x^{r-1})^k \right), \\    
\ell_{\fF,n,k,j}(x) & = \sS^r_0 \left( x^j
      (1 + x + \cdots + x^{r-1})^{n-k} 
			(x + x^2 + \cdots + x^{r-1})^k \right).			
\end{align*}
\end{proposition}

\begin{proof}
The second formula is the special case $j=0$ of 
the third. To verify the first and third formulas
we apply Equations~(\ref{eq:pnk}),~(\ref{eq:lnk-def}) 
and~(\ref{eq:Enr}) to compute that 
\begin{align*} 
p_{\fF,n,k}(x) & = \sum_{i=0}^k {k \choose i} 
   (x-1)^i p_{\fF,n-i,0}(x) \\ & = 
\sum_{i=0}^k {k \choose i} (x-1)^i \sS^r_0 \left( 
(1 + x + \cdots + x^{r-1})^{n-i} \right) \\ 
    & = \sS^r_0 \left( \, \sum_{i=0}^k {k \choose i} 
(x^r-1)^i (1 + x + \cdots + x^{r-1})^{n-i} 
                     \right) \\ & = \sS^r_0 
\left( (1 + x + \cdots + x^{r-1})^{n-k} 
			(x + x^2 + \cdots + x^r)^k \right) \\ & =
\sS^r_0 \left( x^k 
      (1 + x + \cdots + x^{r-1})^n \right)
\end{align*}
and
\begin{align*} 
\ell_{\fF,n,k,j}(x) & = \sum_{i=0}^k (-1)^i 
      {k \choose i} p_{\fF,n-i,j}(x) \\ & = 
\sum_{i=0}^k (-1)^i {k \choose i} \sS^r_0 \left( 
x^j (1 + x + \cdots + x^{r-1})^{n-i} \right) 
    \\ & = \sS^r_0 \left( \, \sum_{i=0}^k (-1)^i 
           {k \choose i} 
x^j (1 + x + \cdots + x^{r-1})^{n-i} \right) 
    \\ & = \sS^r_0 \left( x^j 
		       (1 + x + \cdots + x^{r-1})^{n-k} 
			     (x + x^2 + \cdots + x^{r-1})^k \right)
\end{align*}
and the proof follows.
\end{proof}

\begin{proof}[Proof of Theorem~\ref{thm:esd}.]
We follow the steps of the proof of 
Theorem~\ref{thm:sd}. By 
Proposition~\ref{prop:esdr} and 
Equation~(\ref{eq:Sx}), and since $r \ge n$, we have 
\[ \ell_{\fF,n,k,j}(x) = \begin{cases}
  \sS^r_0 (f(x)), & \text{if $j = 0$} \\
  x \sS^r_{r-j} (f(x)), & \text{if $j \ge 1$} 
\end{cases} \]
for $j \in \{0, 1,\dots,n-k\}$, where 
\[ f(x) = x^k (1 + x + \cdots + x^{r-1})^{n-k} 
   (1 + x + \cdots + x^{r-2})^k. \]
This formula and Lemma~\ref{lem:zhang} imply 
that $(\ell_{\fF,n,k,j}(x))_{0 \le j \le n-k}$ 
is an interlacing sequence of real-rooted 
polynomials for all $0 \le k \le n$. In 
particular, $\ell_{\fF,n,k}(x))$ is real-rooted 
for all $0 \le k \le n$. In view of 
Theorem~\ref{thm:uniform}, it suffices to show
that $E_{n,r}(x) \preceq \ell_{\fF,n,k,j}(x))$ 
for all $0 \le k, j \le n$ with $k+j \le n$.

We now claim that the sequence 
$(\ell_{\fF,n,k}(x))_{0 \le k \le n}$ is 
interlacing for every $n \in \NN$ as well. 
We set 
\[ g(x) = (1 + x + \cdots + x^{r-1})^{n-k-1} 
   (x + x^2 + \cdots + x^{r-1})^k. \]
By Proposition~\ref{prop:esdr}, we have 
\begin{align*}
  \ell_{\fF,n,k}(x) & = \sS^r_0 \left( 
  (1 + x + \cdots + x^{r-1}) g(x) \right) = 
	\sS^r_0 (g(x)) + x \sum_{i=1}^{r-1} \sS^r_i 
	(g(x)) \\ 
	\ell_{\fF,n,k+1}(x) & = \sS^r_0 \left( 
  (x + x^2 + \cdots + x^{r-1}) g(x) \right) = 
	x \sum_{i=1}^{r-1} \sS^r_i (g(x))
\end{align*}
for $0 \le k \le n-1$. Since the sequence 
$(\sS^r_{r-j} (g(x)))_{1 \le j \le r}$ is 
interlacing by Lemma~\ref{lem:zhang}, we have 
\[ \sS^r_0 (g(x)) \preceq x \sum_{i=1}^{r-1} 
   \sS^r_i (g(x)) \]
and hence $\ell_{\fF,n,k}(x) \preceq
\ell_{\fF,n,k+1}(x)$ for every $k \in \{0, 
1,\dots,n-1\}$. Since $\ell_{\fF,n,n}(x) = \ell_V
(\esd_r(2^V),x)$ has symmetric coefficients, with 
center of symmetry $n/2$, we conclude that 
\begin{align*} 
E_{n,r}(x) & = \ell_{\fF,n,0}(x) \preceq 
\ell_{\fF,n,1}(x) \preceq \cdots \preceq 
\ell_{\fF,n,n}(x) \\ & = \iI_n (\ell_{\fF,n,n}(x)) 
\preceq \cdots \preceq \iI_n (\ell_{\fF,n,1}(x)) 
\preceq \iI_n (\ell_{\fF,n,0}(x)) \\ 
       & = \iI_n(E_{n,r}(x)).
\end{align*}
Since $E_{n,r}(x) \preceq \iI_n(E_{n,r}(x))$ by 
\cite[Corollary~7.3]{Ath22} and 
\cite[Theorem~2.7]{BS21}, our claim follows by an 
application of Lemma~\ref{lem:trans}. We now 
recall from Proposition~\ref{prop:lnkj} (b) that 
$\ell_{\fF, n,k,n-k-j}(x) = \iI_n 
(\ell_{\fF,n,k,j}(x))$ for all $0 \le j \le n-k$ 
and conclude that 
\begin{align*} 
E_{n,r}(x) & \preceq \ell_{\fF,n,k}(x) = 
\ell_{\fF,n,k,0}(x) \preceq \ell_{\fF,n,k,1}(x) 
\preceq \cdots \preceq \ell_{\fF,n,k,n-k}(x) \\ 
& = \iI_n (\ell_{\fF,n,k,0}(x)) = \iI_n 
(\ell_{\fF,n,k}(x)) \\ & \preceq \iI_n 
(E_{n,r}(x))
\end{align*}
for every $k \in \{0, 1,\dots,n\}$. A new 
application of Lemma~\ref{lem:trans} shows that
$E_{n,r}(x) \preceq \ell_{\fF,n,k,j}(x)$ for all 
$0 \le j \le n-k$ and the proof follows.
\end{proof}

The fact that $(\ell_{\fF,n,k}(x))_{0 \le k \le n}$ 
is an interlacing sequence of real-rooted polynomials
for every $n \in \NN$, which we have proven, confirms 
part (b) of \cite[Conjecture~7.1]{Ath23+} in the 
special case of edgewise subdivisions.
\begin{example} \label{ex:esdr} \rm
Let $\Gamma$ be the stellar subdivision of $2^V$
on the face $V$ (so that $\Gamma$ has exactly 
one vertex not in $V$, with carrier $V$). Then,
\[ h(\Gamma_F, x) = \begin{cases}
   1 + x + \cdots + x^{n-1}, & \text{if $F = V$} \\
   1, & \text{otherwise} \end{cases} \]
for $F \subseteq V$ and
\[ h(\esd_r(\Gamma)_F, x) = \begin{cases}
   \sS^r_0 \left( (1 + x + \cdots + x^{n-1})
	 (1 + x + \cdots + x^{r-1})^n \right), & 
	          \text{if $F = V$} \\ 
	 \sS^r_0 \left( (1 + x + \cdots + x^{r-1})^{|F|}
	 \right), & \text{otherwise.} \end{cases} \]
One can compute that 
\[ \ell_V(\esd_r(\Gamma),x) = \sS^r_0 \left( 
   (x + x^2 + \cdots + x^{n-1}) 
	 (1 + x + \cdots + x^{r-1})^n + 
	 (x + x^2 + \cdots + x^{r-1})^n \right). \]
For $r=2$ and $n=6$ we have 
$\ell_V(\esd_r(\Gamma),x) = 7x + 42x^2 + 63x^3 + 
42x^4 + 7x^5$, which has two nonreal complex 
roots. We note that neither $\Gamma$ nor 
$\esd_r(\Gamma)$ is a flag triangulation of 
$2^V$.
\end{example}

\section*{Acknowledgments}

The author wishes to thank the anonymous referee 
for their thoughtful comments.

\end{document}